\documentclass[11pt,a4paper]{amsart}
\usepackage[centering,left=3cm,right=3cm]{geometry}
\usepackage{amsfonts,amscd,amssymb,amsmath,amsthm,mathrsfs,xcolor,lscape,longtable}
\usepackage{microtype}
\usepackage[Algorithm]{algorithm}
\usepackage[noend]{algpseudocode}
\usepackage[colorlinks,linkcolor=blue,anchorcolor=blue,citecolor=blue,backref=page]{hyperref}
\usepackage{url}
\usepackage{enumitem}
\usepackage{adjustbox}
\usepackage{color}
\usepackage{graphics,epsfig}
\usepackage{float}
\usepackage{epstopdf}
\usepackage[utf8]{inputenc}
\usepackage{hyperref}
\usepackage{calc, amscd}
\usepackage[T1]{fontenc}
\usepackage{pdfpages}
\usepackage{mathtools} 
\usepackage{scalerel,stackengine}
\hypersetup{breaklinks=true}
\usepackage{multirow,pbox,lipsum}
\usepackage[thinlines]{easytable}
\usepackage{todonotes}
\usepackage{phaistos}

 \usepackage[all,cmtip]{xy}\usepackage{tikz}\usetikzlibrary{arrows}
\usetikzlibrary{matrix}\usetikzlibrary{calc}
\usepackage{tikz-cd,xr,multicol,microtype}

\newtheorem{thm}{Theorem}

\newtheorem{defn}[thm]{Definition}
\newtheorem{cor}[thm]{Corollary}

\theoremstyle{definition}
\newtheorem{rmk}{Remark}
\theoremstyle{definition}

\numberwithin{equation}{section}

\begin{document}
\date{\today}
\title[Thin Monodromy in  $\mathrm{O}(5)$]{Thin Monodromy in $\mathrm{O}(5)$}
\author{Jitendra Bajpai}
\address{Max Planck Institute for Mathematics, Vivatsgasse 7, 53111 Bonn, Germany}
\email{jitendra@mpim-bonn.mpg.de}
\curraddr{Department of Mathematical and Statistical Sciences, University of Alberta, Canada}
\email{jbajpai@ualberta.ca}
\author{Martin Nitsche}
\address{Institute of Algebra and Geometry, Karlsruhe Institute of Technology, Germany}
\email{martin.nitsche@kit.edu}


\begin{abstract}
This article studies the orthogonal hypergeometric groups of degree five.
We establish the thinness of 12 out of the 19 hypergeometric groups of type $O(3,2)$ from~\cite[Table~6]{BS}.
Some of these examples are associated with Calabi-Yau 4-folds. 
We also establish the thinness of 9 out of the 17 hypergeometric groups of type $O(4,1)$ from~\cite{FMS}, where the thinness of 7 other cases was already proven. The $O(4,1)$ type groups were predicted to be all thin and our result leaves just one case open.
\end{abstract}

\maketitle


\section{Introduction}

The hypergeometric differential equation of order $n$ with parameters $\alpha,\beta\in\mathbb{Q}^n$ is defined on $\mathbb{CP}^1\setminus\{0,1,\infty\}$ by
\[
\big[z(\theta+\alpha_{1})\cdot\ldots\cdot(\theta+\alpha_{n})-(\theta+\beta_{1}-1)\cdot\ldots\cdot(\theta+\beta_{n}-1)\big]u(z)=0,
\]
where $\theta=z\frac{d}{dz}$. The monodromy action of the fundamental group on the local solution space $V\cong\mathbb{C}^n$ defines a group representation $\rho\colon\pi_1(\mathbb{CP}^1\setminus\{0,1,\infty\})\to\mathrm{GL}(V)$
and its image is called the hypergeometric group $\Gamma(\alpha,\beta)$.
If the polynomials
\begin{align*}
f(x) & \vcentcolon=\prod_{j=1}^{n}(x-e^{2\pi i\alpha_{j}})=x^n+a_{n-1}x^{n-1}+\cdots+a_1x+a_0,\\
g(x) & \vcentcolon=\prod_{j=1}^{n}(x-e^{2\pi i\beta_{j}})=x^n+b_{n-1}x^{n-1}+\cdots+b_1x+b_0
\end{align*}
do not have a common root, then by a result of Levelt (see \cite[Thm.~3.5]{BH}) there exists a basis for $V$ such that the images under $\rho$ of the loops around $0,1,\infty\in\mathbb{CP}^1$ are represented by the matrices
\begin{align}\label{eq:companion-matrices}
g_{\infty} & \mapsto A=\begin{psmallmatrix}
0&0&\cdots&0&-a_0\\1&0&\cdots&0&-a_1\\0&1&\cdots&0&-a_2\\ \vdots&\vdots&\ddots&\vdots&\vdots\\ 0&0&\cdots&1&-a_{n-1}
\end{psmallmatrix}, & g_{0} & \mapsto B^{-1}=\begin{psmallmatrix}
0&0&\cdots&0&-b_0\\1&0&\cdots&0&-b_1\\0&1&\cdots&0&-b_2\\ \vdots&\vdots&\ddots&\vdots&\vdots\\ 0&0&\cdots&1&-b_{n-1}
\end{psmallmatrix}^{-1}, & g_{1} & \mapsto A^{-1}B.
\end{align}
When $f$ and $g$ are products of cyclotomic polynomials, $\Gamma(f,g)\vcentcolon=\Gamma(\alpha,\beta)=\langle A,B\rangle$ is a subgroup of $\mathrm{GL}(n,\mathbb{Z})$.
It is a question due to Sarnak~\cite{Sa14}, which of these groups are arithmetic and which thin.

\begin{defn}
A hypergeometric group $\Gamma(f,g)$ is called arithmetic if it has finite index in $\mathrm{G}(\mathbb{Z})$, and thin if it has infinite index in $\mathrm{G}(\mathbb{Z})$, where $\mathrm{G}$ is the Zariski closure of $\Gamma(f,g)$.
\end{defn}

Sarnak's question has witnessed many interesting developments.
In particular, Venkataramana~\cite{V17} has constructed 11 infinite families of higher rank arithmetic orthogonal hypergeometric groups, and Fuchs--Meiri--Sarnak~\cite{FMS} have constructed 7 infinite families of hyperbolic thin orthogonal hypergeometric groups.
For a detailed account on recent progress see the introduction of~\cite{bdss}.
The problem is usually broken into different parts corresponding to the Zariski closure of $\Gamma(f,g)$, which -- setting aside the imprimitive case \cite[Def.~5.1]{BH} -- can be either symplectic, orthogonal, or finite.
In this article, we consider the orthogonal case in degree $n=5$.

From the results of~\cite{BH} it follows that, up to scalar shift (see~\cite[Def.~5.5]{BH}), there are exactly $77$ cases that satisfy the above conditions. These are discussed in~\cite{BS}.
Four of these cases correspond to finite monodromy groups and 17 cases correspond to monodromy groups for which the Zariski closure is of type $O(4,1)$, hence has real rank one.
It follows from \cite{FMS} that 7 of the type $O(4,1)$ cases are thin. The other 10 cases, which were still open, are listed in Table~\ref{ta:O41-thin} and Table~\ref{ta:O41-open} below.

For the remaining 56 cases the Zariski closure is $O(3,2)$, with real rank two, see~\cite{BS}.
Out of these 56 cases, 37 have already been proven to be arithmetic: 11 cases~\cite[Table~2]{BS} by the results of Venkataramana~\cite{V17}, 2 cases~\cite[Table~3]{BS} by the results of Singh~\cite{S15O}, 23 cases~\cite[Table~4]{BS} by Bajpai--Singh~\cite{BS}, and 1 case by Bajpai--Singh--Singh~\cite{BSS}.
There remain 19 cases of type $O(3,2)$ whose arithmeticity or thinness was still undetermined. These cases are listed in Table~\ref{ta:O32-thin} and Table~\ref{ta:O32-open} below.
\vspace{0.2cm}

Of particular relevance among the 56 examples of type $O(3,2)$ are the 14 pairs with \emph{maximally unipotent monodromy}, that is, those hypergeometric groups $\Gamma(f,g)$ where the polynomial $f$ is associated to $\alpha=\left(0,0,0,0,0\right)$.

It is well known that in dimension $n=4$, the 14 symplectic hypergeometric groups with maximally unipotent monodromy emerge as images of monodromy representations arising from Calabi-Yau 3-folds, see~\cite{DM06}. It was shown~\cite{S15S, SV, BT} that exactly half of these groups are arithmetic and half are thin.
Similarly, for $n=6$ it is expected that many of the 40 symplectic hypergeometric groups with maximally unipotent monodromy arise from Calabi-Yau 5-folds~\cite{GMP,LTY}. Out of these groups, 23 are arithmetic and 17 are thin~\cite{bdss,BDN21,BDN22}.

Now, for $n=5$, it is known that at least some of the 14 orthogonal hypergeometric groups with maximally unipotent monodromy arise from Calabi-Yau 4-folds; see~\cite[Sec.~3.9.3]{BKL} for a detailed account. One of the purposes of this work is to investigate the dichotomy between arithmetic and thin monodromy among the 14 groups with maximally unipotent monodromy.
Out of these cases, 2 have been shown to be arithmetic by Singh~\cite{S15O}, and in this article we prove that 9 of them are thin. The other 3 cases remain open. But our result means that in dimension $n=5$ \emph{more} than half of the hypergeometric groups with maximally unipotent monodromy are thin, in contrast to the symplectic cases in both dimension $n=4$ and $n=6$.

One interesting example is example~\ref{T7} in Table~\ref{ta:O32-thin}. This is the sextic case, where $\alpha=\left(0,0,0,0,0\right)$, $\beta=\big(\frac{1}{6},\frac{2}{6},\frac{3}{6},\frac{4}{6},\frac{5}{6}\big)$. We refer to~\cite[Sec.~6]{KP} for an account of this particular case.

\subsection{Results}
With a ping pong argument very similar to that in~\cite{BT} and~\cite{BDN21} we obtain the following result.

\begin{thm}\label{thm:main-result}
Let $G$ be any one of the groups in Tables~\ref{ta:O32-thin} and \ref{ta:O41-thin}, let $A,B$ be the generators from Equation~\ref{eq:companion-matrices} and let $T=BA^{-1}$. Then

\begin{itemize}
\item $G=\langle\pm T\rangle\ast_{\pm I}\langle B\rangle$ if $B^k=-I$ for some $k\in\mathbb{N}$, and
\item $G=\langle T\rangle\ast\langle B\rangle$ otherwise.
\end{itemize}
\end{thm}

\begin{cor}
The hypergeometric groups in Tables~\ref{ta:O32-thin} and~\ref{ta:O41-thin} are thin.
\end{cor}
\begin{proof}
Both $O(3,2)$ and $O(4,1)$ have trivial first $L^2$-Betti numbers \cite[Ex.\ 1.6]{G}, and by the Proportionality Principle \cite[Cor.\ 0.2, Thm.\ 6.3]{G} the same is true for lattices in these groups. But the (amalgamated) free products in Theorem~\ref{thm:main-result} have non-trivial first $L^2$-Betti number~\cite[Thm.\ A.1]{BDJ}.
\end{proof}

\begin{rmk}
Simion Filip has informed us that he has also proven thinness of the groups in Table~\ref{ta:O32-thin} (among other results), using tools from Hodge theory \cite{Filip}.
\end{rmk}

{\renewcommand{\arraystretch}{2}
\begin{table}[htbp]
\small
\addtolength{\tabcolsep}{-1pt}
\caption{List of the $12$ monodromy groups for which the associated quadratic form $\mathrm{Q}$ has signature $(3,2)$ and thinness is shown in this article.}\label{ta:O32-thin}
\newcounter{rownum-0}
\setcounter{rownum-0}{0}
\centering
\begin{tabular}{|c|c|c|c|c|c|}
\hline
  No. &$\alpha$ & $\beta$ & No. &$\alpha$ & $\beta$\\ 
  \hline  
  \refstepcounter{rownum-0}\arabic{rownum-0}\label{T1} & $\big( 0,0,0,0,0\big)$  & $\big(\frac{1}{2},\frac{1}{2},\frac{1}{2},\frac{1}{2},\frac{1}{2}\big)$ &
  \refstepcounter{rownum-0}\arabic{rownum-0}\label{T2} & $\big( 0,0,0,0,0\big)$  & $\big(\frac{1}{2},\frac{1}{2},\frac{1}{2},\frac{1}{3},\frac{2}{3}\big)$\\ 
  \hline
  \refstepcounter{rownum-0}\arabic{rownum-0}\label{T3} &  $\big( 0,0,0,0,0\big)$ &  $\big(\frac{1}{2},\frac{1}{2},\frac{1}{2},\frac{1}{4},\frac{3}{4}\big)$&
  \refstepcounter{rownum-0}\arabic{rownum-0}\label{T4} &$\big( 0,0,0,0,0\big)$ & $\big(\frac{1}{2},\frac{1}{2},\frac{1}{2},\frac{1}{6},\frac{5}{6}\big)$\\ 
  \hline
  \refstepcounter{rownum-0}\arabic{rownum-0}\label{T7} & $\big( 0,0,0,0,0\big)$   &$\big(\frac{1}{2},\frac{1}{3},\frac{2}{3},\frac{1}{6},\frac{5}{6}\big)$  &
  \refstepcounter{rownum-0}\arabic{rownum-0}\label{T9} &  $\big( 0,0,0,0,0\big)$  & $\big(\frac{1}{2},\frac{1}{5},\frac{2}{5},\frac{3}{5},\frac{4}{5} \big)$\\
  \hline
  \refstepcounter{rownum-0}\arabic{rownum-0}\label{T10} &$\big( 0,0,0,0,0\big)$ & $\big(\frac{1}{2},\frac{1}{8},\frac{3}{8},\frac{5}{8}, \frac{7}{8} \big)$&
  \refstepcounter{rownum-0}\arabic{rownum-0}\label{T11} & $\big( 0,0,0,0,0\big)$  &$\big(\frac{1}{2},\frac{1}{10},\frac{3}{10},\frac{7}{10}, \frac{9}{10} \big)$ \\
  \hline
  \refstepcounter{rownum-0}\arabic{rownum-0}\label{T12} & $\big( 0,0,0,0,0\big)$ & $\big(\frac{1}{2},\frac{1}{12},\frac{5}{12},\frac{7}{12}, \frac{11}{12} \big)$&
   \refstepcounter{rownum-0}\arabic{rownum-0}\label{T13} &$\big(0,0,0,\frac{1}{4},\frac{3}{4}\big)$ &$\big(\frac{1}{2},\frac{1}{2},\frac{1}{2},\frac{2}{3},\frac{2}{3}\big)$\\ 
   \hline
   \refstepcounter{rownum-0}\arabic{rownum-0}\label{T15} &$\big(0,0,0,\frac{1}{6},\frac{5}{6}\big)$ &$\big(\frac{1}{2},\frac{1}{2},\frac{1}{2},\frac{2}{3},\frac{2}{3}\big)$&\refstepcounter{rownum-0}\arabic{rownum-0}\label{unknown-BS20} & $\big(0,0,0,\frac{1}{6},\frac{5}{6}\big)$  & $\big(\frac{1}{2},\frac{1}{5},\frac{2}{5},\frac{3}{5},\frac{4}{5} \big)$ \\
   \hline
   \end{tabular}
\end{table}}

{\renewcommand{\arraystretch}{2}
\begin{table}[htbp]
\small
\addtolength{\tabcolsep}{-1pt}
\caption{List of the $7$ monodromy groups for which the associated quadratic form $\mathrm{Q}$ has signature $(3,2)$ and arithmeticity or thinness remains unknown.}\label{ta:O32-open}
\setcounter{rownum-0}{0}
\centering
\begin{tabular}{|c|c|c|c|c|c|}
\hline
 No. &$\alpha$ & $\beta$ & No. &$\alpha$ & $\beta$\\ 
 \hline  
 \refstepcounter{rownum-0}\arabic{rownum-0}\label{O-BM1} & $\big( 0,0,0,0,0\big)$& $\big(\frac{1}{2},\frac{1}{3},\frac{1}{3},\frac{2}{3},\frac{2}{3}\big)$ &\refstepcounter{rownum-0}\arabic{rownum-0}\label{O-BM2} & $\big( 0,0,0,0,0\big)$  &$\big(\frac{1}{2},\frac{1}{3},\frac{2}{3},\frac{1}{4},\frac{3}{4}\big)$\\ 
 \hline 
 \refstepcounter{rownum-0}\arabic{rownum-0}\label{O-BM3} & $\big( 0,0,0,0,0\big)$&$\big(\frac{1}{2},\frac{1}{4},\frac{3}{4},\frac{1}{6},\frac{5}{6}\big)$
 &\refstepcounter{rownum-0}\arabic{rownum-0}\label{unknown-BS15} &$\big(0,0,0,\frac{1}{4},\frac{3}{4}\big)$ &$\big(\frac{1}{2},\frac{1}{3},\frac{1}{3},\frac{2}{3},\frac{2}{3}\big)$\\ 
 \hline
 \refstepcounter{rownum-0}\arabic{rownum-0}\label{unknown-BS18} &$\big(0,0,0,\frac{1}{6},\frac{5}{6}\big)$ & $\big(\frac{1}{2},\frac{1}{3},\frac{1}{3},\frac{2}{3},\frac{2}{3}\big)$&
 \refstepcounter{rownum-0}\arabic{rownum-0}\label{unknown-BS19} &$\big(0,0,0,\frac{1}{6},\frac{5}{6}\big)$ & $\big(\frac{1}{2},\frac{1}{3},\frac{2}{3},\frac{1}{4},\frac{3}{4}\big)$\\
 \hline
  \refstepcounter{rownum-0}\arabic{rownum-0}\label{unknown-BS21} &$\big(0,\frac{1}{10},\frac{3}{10},\frac{7}{10},\frac{9}{10}\big)$ & $\big(\frac{1}{2},\frac{1}{3},\frac{1}{3},\frac{2}{3},\frac{2}{3}\big)$& & & \\
 \hline
 \end{tabular}
\end{table}}

{\renewcommand{\arraystretch}{1.7}
\begin{table}[htbp]
\small
\addtolength{\tabcolsep}{-1pt}
\caption{List of the $9$ monodromy groups for which the associated quadratic form $\mathrm{Q}$ has signature $(4,1)$ and thinness is shown in this article.}\label{ta:O41-thin}
\setcounter{rownum-0}{0}
\centering
\begin{tabular}{ |c|  c|   c| c| c| c|}
\hline
  No. &$\alpha$ & $\beta$ & No. &$\alpha$ & $\beta$\\ \hline
  \refstepcounter{rownum-0}\arabic{rownum-0}\label{thin1-rank1} & $\big( 0,0,0,\frac{1}{3},\frac{2}{3}\big)$ & $\big(\frac{1}{2},\frac{1}{2},\frac{1}{2},\frac{1}{4},\frac{3}{4}\big)$  &\refstepcounter{rownum-0}\arabic{rownum-0}\label{thin2-BS13} & $\big( 0,0,0,\frac{1}{3},\frac{2}{3}\big)$  & $\big(\frac{1}{2},\frac{1}{2},\frac{1}{2},\frac{1}{6},\frac{5}{6}\big)$\\ \hline
  \refstepcounter{rownum-0}\arabic{rownum-0}\label{thin3-rank1} & $\big( 0,0,0,\frac{1}{3},\frac{2}{3}\big)$ &$\big(\frac{1}{2},\frac{1}{5},\frac{2}{5},\frac{3}{5},\frac{4}{5} \big)$&\refstepcounter{rownum-0}\arabic{rownum-0}\label{thin4-rank1} &$\big(0,0,0,\frac{1}{4},\frac{3}{4}\big)$ &$\big(\frac{1}{2},\frac{1}{3},\frac{2}{3},\frac{1}{6},\frac{5}{6}\big)$\\ \hline
  \refstepcounter{rownum-0}\arabic{rownum-0}\label{thin5-rank1} &$\big(0,0,0,\frac{1}{4},\frac{3}{4}\big)$ & $\big(\frac{1}{2},\frac{1}{5},\frac{2}{5},\frac{3}{5},\frac{4}{5} \big)$&\refstepcounter{rownum-0}\arabic{rownum-0}\label{thin6-BS16} &$\big(0,0,0,\frac{1}{4},\frac{3}{4}\big)$ & $\big(\frac{1}{2},\frac{1}{8},\frac{3}{8},\frac{5}{8}, \frac{7}{8} \big)$\\ \hline
  \refstepcounter{rownum-0}\arabic{rownum-0}\label{thin7-rank1} &$\big(0,0,0,\frac{1}{4},\frac{3}{4}\big)$ & $\big(\frac{1}{2},\frac{1}{12},\frac{5}{12},\frac{7}{12}, \frac{11}{12} \big)$&\refstepcounter{rownum-0}\arabic{rownum-0}\label{thin8-rank1} &$\big(0,0,0,\frac{1}{6},\frac{5}{6}\big)$  &$\big(\frac{1}{2},\frac{1}{8},\frac{3}{8},\frac{5}{8}, \frac{7}{8} \big)$\\ \hline
  \refstepcounter{rownum-0}\arabic{rownum-0}\label{thin9-rank1} &$\big(0,0,0,\frac{1}{6},\frac{5}{6}\big)$  & $\big(\frac{1}{2},\frac{1}{12},\frac{5}{12},\frac{7}{12}, \frac{11}{12} \big)$& & & \\ \hline
  \end{tabular}
\end{table}}

{\renewcommand{\arraystretch}{1.7}
\begin{table}[htbp]
\small
\addtolength{\tabcolsep}{-1pt}
\caption{The one monodromy group for which the associated quadratic form $\mathrm{Q}$ has signature $(4,1)$ and arithmeticity or thinness remains unknown.}
\setcounter{rownum-0}{0}
\centering
\begin{tabular}{ |c|  c|   c| c| c| c|}
\hline
  No. &$\alpha$ & $\beta$  \\ \hline
  \refstepcounter{rownum-0}\arabic{rownum-0}\label{unknown-rank1} & $\big(0,\frac{1}{10},\frac{3}{10},\frac{7}{10},\frac{9}{10}\big)$& $\big(\frac{1}{2},\frac{1}{5},\frac{2}{5},\frac{3}{5},\frac{4}{5} \big)$\\ \hline
  \end{tabular}
\label{ta:O41-open}
\end{table}}

\section{The ping-pong setup}\label{sec:ping-pong}
Our proof is very similar to that in~\cite{BDN21}, which, in turn, is an adaptation of the methods of Brav and Thomas~\cite{BT}. Using the notation of Theorem~\ref{thm:main-result}, our goal is to apply the following version of the ping pong lemma to the case where $G_1=\langle T\rangle$, $G_2=\langle B\rangle$, $H=\{I\}$, respectively to $G_1=\langle\pm T\rangle$, $G_2=\langle B\rangle$, $H=\{\pm I\}$.

\begin{thm}[see \cite{LS77}, Prop.~III.12.4]
Let $G$ be a group generated by two subgroups $G_1$ and $G_2$, whose intersection $H$ has index $>2$ in $G_{1}$ or $G_{2}$. Suppose that $G$ acts on a set $W$, and suppose that there are disjoint non-empty subsets $X, Y\subset W$, such that $(G_1\setminus H) Y \subseteq X$ and $(G_2\setminus H) X \subseteq Y$, with $H Y \subseteq Y$ and $H X \subseteq X$. Then $G = G_1 \ast_{H} G_2$.
\end{thm}

We apply the ping pong lemma to the canonical action of $G$ on $\mathbb{R}^5$.
The two halves $X$, $Y$ of the ping pong table will both decompose into a union of open cones (with the origin removed) that is invariant under multiplication with $-I$. The table halves are constructed from a single non-empty convex cone $F$ that we will explicitly provide for each case. The exact construction depends on the order of $B$. Note that $T$ has always order $2$.

If $B$ has finite order, we set
\[X=F\cup -F,\qquad Y=\bigcup_{B^i\neq\pm I}B^i F\cup -B^i F.\]
To verify that we get a valid ping pong, we then have to check that $X$ and $Y$ are disjoint and that $T$ maps $Y$ into $X$. The other conditions on the ping pong table are automatically satisfied.

If $B$ has infinite order, we let
\[X=F_0\cup -F_0,\qquad F_0=F\cap -EF,\]
for a certain matrix $E$ that satisfies $E^2=EBE^{-1}B=ETE^{-1}T=I$. We also let $\eta=\min\,\{i>0\mid (B^i-I)^5=0\}$ and let $Y$ be the finite union
\[Y=Y^+\cup Y^-,\qquad Y^+=\bigcup_{1\leq i\leq\eta}B^i F\cup -B^i F,\qquad Y^-=EY^+.\]
To verify that we get a valid ping pong, we again have to check that $X$ and $Y$ are disjoint and that $T$ maps $Y$ into $X$. The condition that $B^{i\neq 0}$ maps $X$ into $Y$ is no longer automatic. To verify this in a finite number of steps, we check that $B$ maps both $X$ and $Y^+$ into $Y^+$, and that $B^{-1}$ maps both $X$ and $Y^-$ into $Y^-$, see the right diagram in Figure~\ref{fig:ping-pong}.

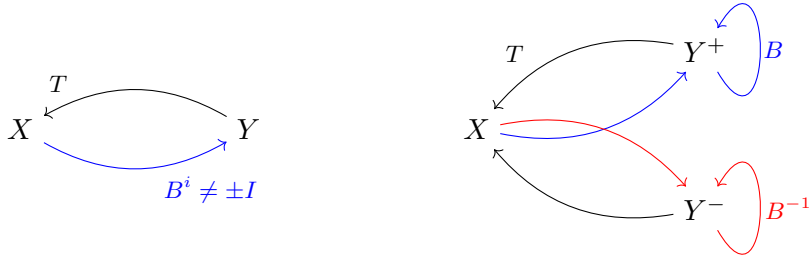
\begin{figure}[htbp]
\begin{minipage}{12cm}
\centering
\begin{tikzpicture}[scale=1.5]
\node (xa) at (-5,0) {$X$};
\node (y) at (-3,0) {$Y$};
\draw[blue,->,bend right]   (xa) edge (y);
\draw[black,->,bend right] (y) edge  (xa);
\node[xshift=0.5cm,yshift=0.6cm] at (xa) {\footnotesize\color{black}$T$};
\node[xshift=-0.5cm,yshift=-0.8cm] at (y) {\footnotesize\color{blue}$B^i\neq\pm I$};
\node (x) at (-1,0) {$X$};
\node (yp) at (1,0.7) {$Y^{+}$};
\node (ym) at (1,-0.7) {$Y^{-}$};
\draw[blue,->,bend right]   (x) edge (yp);
\draw[red,->,bend left]     (x) edge (ym);
\draw[black,->,bend right] (yp) edge  (x);
\draw[black,->,bend left]  (ym) edge  (x);
\node[xshift=0.5cm,yshift=1cm] at (x) {\footnotesize\color{black}$T$};
\node[xshift=0.9cm,yshift=0cm] at (yp) {\footnotesize\color{blue}$B$};
\node[xshift=1.1cm,yshift=0cm] at (ym) {\footnotesize\color{red}$B^{-1}$};
\draw[blue,->]  (yp) edge [out=300,in=60,distance=10mm]   (yp);
\draw[red,->]  (ym) edge [out=300,in=60,distance=10mm]   (ym);
\end{tikzpicture}
\caption{The ping pong table when $B$ has finite (left) and infinite order (right)}\label{fig:ping-pong}
\end{minipage}
\end{figure}

\section{Computation}
The computer calculations for constructing and verifying a working ping-pong table are quite similar to the symplectic case covered in~\cite{BDN21}, to which we refer the reader for a more detailed discussion.

The main difference in the orthogonal case is that $T$ is a reflection, not a transvection. The initial guess for the cone $F$ is now obtained from $\lim_{i\to\infty}(TB)^i$ instead of $\lim_{i\to\infty}T^i$. Then, as before, we iteratively expand $F$ by applying the ping-pong rules of Figure~\ref{fig:ping-pong}. Special care is necessary to make this process stop after finitely many steps.
The construction and verification of the ping-pong tables are not computation intensive, but too laborious to do manually.

For the remaining cases listed in Tables~\ref{ta:O32-open} and \ref{ta:O41-open}, our ping pong approach did not succeed, and as in~\cite{BDN21} we checked that the ping pong setup of Section~\ref{sec:ping-pong} must fail for any choice of $F$.
A more complicated ping pong could still work for these cases, but we suspect that many, if not all, of the remaining cases are, in fact, arithmetic.

\section{Verification}

Below, we list for all cases in Tables~\ref{ta:O32-thin} and \ref{ta:O41-thin} a precomputed cone $F$ in the form of a column matrix containing its spanning rays.
The following SageMath computer code, adapted from \cite{BDN21}, can be used to verify that the resulting ping pong tables satisfy the conditions from Section~\ref{sec:ping-pong} in each case.
The cones are implemented by the library class \verb|ConvexRationalPolyhedralCone|, which does all calculations with exact arithmetic in rational numbers.

\vspace{0.5cm}
\small
\begin{verbatim}
# The ping-pong table halves, which consist of unions of open cones,
# are represented in this code as a list of the closures of these cones.
# An open cone is non-empty if its closure has the same dimension as the
# ambient vector space. Two open cones are disjoint if the intersection
# of their closures has less than full dimension.

from itertools import count

def companion_matrix(polynomial):
  return block_matrix([[block_matrix([[matrix(1,4)],[matrix.identity(4)]]),
                        -matrix(polynomial.list()[:-1]).T]])

# check if the two sets of open cones are disjoint
def are_disjoint(CC, DD):
  return all(C.intersection(D).dim() < 5 for C in CC for D in DD)

# check if the first set of cones is contained in the second
def contained_in(CC, DD):
  return all(any(is_subcone(C, D) for C in CC) for D in DD)

# check if the first cone is contained in the second
def is_subcone(C, D):
  return all(ray in D for ray in C.rays())

# apply the linear transformation L to the set of cones
def transform_set(CC, L):
  return [Cone((L*C.rays().column_matrix()).T) for C in CC]

def verify_finite_order(B, T, F):
  if not T^2 == 1:
    return False # we assume that T has order 2
  eta = next(i for i in count(1) if (B^i-1)^5 == 0)
  if not B^eta == 1:
    return False # B not of finite order
  if not F.dim() == 5:
    return False # F has empty interior
  X = [F] + transform_set([F],-1)
  Y = [C for i in [1..eta-1] for C in transform_set(X, B^i) if not B^i == -1]
  if not are_disjoint(X, Y):
    return False # ping pong table halves not disjoint
  if not contained_in(transform_set(Y, T), X):
    return False # T does not map Y into X
  return True    # remaining conditions automatic, ping pong works

def verify_infinite_order(B, T, F):
  if not T^2 == 1:
    return False # we assume that T has order 2
  if not F.dim() == 5:
    return False # F has empty interior
  E = B*Permutation([5,4,3,2,1]).to_matrix()
  F0 = F.intersection(transform_set([F],-E)[0])
  X = [F0] + transform_set([F0],-1)
  eta = next(i for i in count(1) if (B^i-1)^5 == 0)
  Yplus = [C for i in [1..eta] for sgn in [1,-1]
                               for C in transform_set([F], sgn*B^i)]
  Yminus = transform_set(Yplus, E)
  if not are_disjoint(X, Yplus+Yminus):
    return False # ping pong table halves not disjoint
  if not contained_in(transform_set(Yplus+Yminus, T), X):
    return False # T does not map Y into X
  if not contained_in(transform_set(X+Yplus, B), Yplus):
    return False # B does not map both X and Y^+ into Y^+
  if not contained_in(transform_set(X+Yminus, B.inverse()), Yminus):
    return False # B^-1 does not map both X and Y^- into Y^-
  return True    # remaining conditions automatic, ping pong works

# example usage to verify Case 7 of type O(3,2)
A=companion_matrix(cyclotomic_polynomial(1)^5)
B=companion_matrix(cyclotomic_polynomial(2)*cyclotomic_polynomial(8))
M=matrix([[ -8,  -8, -1, -1,  0,  0, -1,  -3,  -3, -1],
          [  5,  35,  0,  4,  1,  1,  5,  14,   4,  2],
          [-45, -59, -4, -6, -3, -2, -9, -25, -21, -8],
          [ 59,  45,  6,  4,  3,  2,  8,  21,  25,  9],
          [-43, -13, -5, -1, -1, -1, -3,  -7, -17, -6]])
print(verify_finite_order(B, B*A.inverse(), Cone(M.T)))
\end{verbatim}
\normalsize

\subsection{Cases of type $O(3,2)$}
\subsubsection*{Case 1} $\alpha= \left(0,0,0,0,0\right)$, $\beta=\left(\frac{1}{2},\frac{1}{2},\frac{1}{2},\frac{1}{2},\frac{1}{2}\right)$, and $B$ has infinite order.
\[\adjustbox{max width=\textwidth}{%
$M = \left(\begin{array}{rrrrrrrrrrrrrr}
3 & -1 & -31 & -245 & -5 & -5 & -1 & -5 & -1 & -1 & -7 & 0 & -1 & -8 \\
-22 & -4 & 158 & 1194 & 23 & 24 & -4 & -24 & 4 & -6 & -32 & 1 & -5 & 39 \\
-84 & -6 & -332 & -2292 & -41 & -44 & -14 & -44 & -6 & -10 & -58 & -3 & -9 & -85 \\
-90 & -4 & 226 & 2118 & 33 & 40 & -4 & -40 & 4 & -10 & -48 & 3 & -9 & 69 \\
-31 & -1 & -245 & -999 & -10 & -15 & -9 & -15 & -1 & -5 & -15 & -1 & -4 & -47
\end{array}\right)$}\]

\subsubsection*{Case 2} $\alpha= \left(0,0,0,0,0\right)$, $\beta=\left(\frac{1}{2},\frac{1}{2},\frac{1}{2},\frac{1}{3},\frac{2}{3}\right)$, and $B$ has infinite order.
\[\adjustbox{max width=\textwidth}{%
$M = \left(\begin{array}{rrrrrrrrrrrrrrr}
-33 & -33 & -31 & -31 & -1 & -1 & -1 & -1 & -1 & -21 & -2 & -11 & -11 & -2 & 0 \\
-37 & 158 & -63 & 142 & -3 & -3 & -2 & 4 & 5 & -73 & 9 & 53 & -6 & 3 & 1 \\
-489 & -288 & -419 & -252 & -4 & -11 & -6 & -6 & -3 & -104 & -15 & -101 & -172 & -31 & -3 \\
57 & 258 & 35 & 202 & -3 & -1 & -4 & 4 & -1 & -83 & 17 & 95 & 24 & 1 & 3 \\
-290 & -95 & -266 & -61 & -1 & -8 & -9 & -1 & -2 & -31 & -11 & -38 & -97 & -17 & -1
\end{array}\right)$}\]

\subsubsection*{Case 3} $\alpha= \left(0,0,0,0,0\right)$, $\beta=\left(\frac{1}{2},\frac{1}{2},\frac{1}{2},\frac{1}{4},\frac{3}{4}\right)$, and $B$ has infinite order.
\[\adjustbox{max width=\textwidth}{%
$M = \left(\begin{array}{rrrrrrrrrrrrrrrrr}
-1 & -13 & -13 & -11 & -11 & -1 & -1 & -1 & -1 & -11 & -2 & -11 & -38 & -38 & -11 & -2 & 0 \\
-2 & -4 & 62 & -12 & 50 & -2 & -1 & 4 & 5 & -32 & 9 & 53 & 179 & -19 & 5 & 5 & 1 \\
-2 & -150 & -112 & -114 & -88 & -8 & -3 & -6 & -3 & -32 & -15 & -101 & -327 & -431 & -139 & -25 & -3 \\
-2 & 60 & 98 & 44 & 70 & 2 & -1 & 4 & -1 & -32 & 17 & 95 & 279 & 175 & 57 & 7 & 3 \\
-1 & -101 & -35 & -83 & -21 & -7 & -8 & -1 & -2 & -21 & -11 & -38 & -95 & -293 & -86 & -15 & -1
\end{array}\right)$}\]

\subsubsection*{Case 4} $\alpha= \left(0,0,0,0,0\right)$, $\beta=\left(\frac{1}{2},\frac{1}{2},\frac{1}{2},\frac{1}{6},\frac{5}{6}\right)$, and $B$ has infinite order.
\[\adjustbox{max width=\textwidth}{%
$M = \left(\begin{array}{rrrrrrrrrrrrrrrrr}
-1 & -13 & -13 & -1 & -1 & -1 & -1 & -2 & -2 & -11 & -38 & -38 & -11 & -2 & -3 & -3 & 0 \\
-1 & -3 & 58 & -1 & 0 & 4 & 7 & -3 & 9 & 53 & 179 & 17 & 16 & 7 & 13 & 8 & 1 \\
0 & -91 & -100 & -5 & 0 & -6 & -3 & 0 & -13 & -101 & -327 & -317 & -108 & -19 & -33 & -24 & -3 \\
-1 & 87 & 78 & 5 & 2 & 4 & -1 & -2 & 17 & 97 & 279 & 289 & 90 & 11 & 21 & 30 & 3 \\
-1 & -84 & -23 & -6 & -9 & -1 & -2 & -3 & -11 & -38 & -93 & -255 & -75 & -13 & -14 & -19 & -1
\end{array}\right)$}\]

\subsubsection*{Case 5} $\alpha= \left(0,0,0,0,0\right)$, $\beta=\left(\frac{1}{2},\frac{1}{3},\frac{2}{3},\frac{1}{6},\frac{5}{6}\right)$, and $B$ has finite order.
\[\adjustbox{max width=\textwidth}{%
$M  = \left(\begin{array}{rrrrrrrrrrrr}
-20 & -20 & -2 & -2 & -1 & -1 & 0 & 0 & -1 & -3 & -3 & -1 \\
15 & 89 & 0 & 1 & 0 & 4 & 1 & 1 & 5 & 14 & 4 & 2 \\
-139 & -153 & -3 & -1 & -5 & -6 & -3 & -2 & -9 & -25 & -24 & -9 \\
133 & 119 & -1 & 1 & 5 & 4 & 3 & 2 & 8 & 21 & 22 & 8 \\
-109 & -35 & -3 & -2 & -5 & -1 & -1 & -1 & -3 & -7 & -17 & -6
\end{array}\right)$}\]

\subsubsection*{Case 6} $\alpha= \left(0,0,0,0,0\right)$, $\beta=\left(\frac{1}{2},\frac{1}{5},\frac{2}{5},\frac{3}{5},\frac{4}{5} \right)$, and $B$ has finite order.
\[\adjustbox{max width=\textwidth}{%
$M =\left(\begin{array}{rrrrrrrrrrrrrrr}
-6 & -6 & -4 & -4 & -1 & -1 & -1 & -2 & -11 & -38 & -38 & -11 & -2 & -1 & 0 \\
-1 & 27 & 2 & 19 & -1 & 4 & 0 & 7 & 16 & 19 & 179 & 53 & 9 & 5 & 1 \\
-49 & -47 & -37 & -34 & -6 & -6 & -1 & -21 & -117 & -355 & -327 & -101 & -15 & -3 & -3 \\
35 & 37 & 26 & 29 & 4 & 4 & 1 & 11 & 79 & 251 & 279 & 95 & 17 & -1 & 3 \\
-39 & -11 & -27 & -10 & -6 & -1 & -7 & -13 & -75 & -255 & -95 & -38 & -11 & -2 & -1
\end{array}\right)$}\]

\subsubsection*{Case 7} $\alpha= \left(0,0,0,0,0\right)$, $\beta=\left(\frac{1}{2},\frac{1}{8},\frac{3}{8},\frac{5}{8}, \frac{7}{8} \right)$, and $B$ has finite order.
\[\adjustbox{max width=\textwidth}{%
$M = \left(\begin{array}{rrrrrrrrrr}
-8 & -8 & -1 & -1 & 0 & 0 & -1 & -3 & -3 & -1 \\
5 & 35 & 0 & 4 & 1 & 1 & 5 & 14 & 4 & 2 \\
-45 & -59 & -4 & -6 & -3 & -2 & -9 & -25 & -21 & -8 \\
59 & 45 & 6 & 4 & 3 & 2 & 8 & 21 & 25 & 9 \\
-43 & -13 & -5 & -1 & -1 & -1 & -3 & -7 & -17 & -6
\end{array}\right)$}\]

\subsubsection*{Case 8} $\alpha= \left(0,0,0,0,0\right)$, $\beta=\left(\frac{1}{2},\frac{1}{10},\frac{3}{10},\frac{7}{10}, \frac{9}{10} \right)$, and $B$ has finite order.
\[\adjustbox{max width=\textwidth}{%
$M = \left(\begin{array}{rrrrrrrrrrrr}
-10 & -10 & -1 & -1 & 0 & 0 & -1 & -3 & -7 & -7 & -3 & -1 \\
15 & 43 & 1 & 4 & 1 & 1 & 5 & 14 & 32 & 14 & 7 & 3 \\
-53 & -71 & -4 & -6 & -3 & -2 & -9 & -25 & -56 & -45 & -21 & -8 \\
71 & 53 & 6 & 4 & 3 & 2 & 8 & 21 & 45 & 56 & 25 & 9 \\
-43 & -15 & -4 & -1 & -1 & -1 & -3 & -7 & -14 & -32 & -14 & -5
\end{array}\right)$}\]

\subsubsection*{Case 9} $\alpha= \left(0,0,0,0,0\right)$, $\beta=\left(\frac{1}{2},\frac{1}{12},\frac{5}{12},\frac{7}{12}, \frac{11}{12}\right)$, and $B$ has finite order.
\[\adjustbox{max width=\textwidth}{%
$M = \left(\begin{array}{rrrrrrrrrrrr}
-12 & -12 & -1 & -1 & 0 & 0 & -1 & -3 & -7 & -7 & -3 & -1 \\
5 & 51 & 0 & 4 & 1 & 1 & 5 & 14 & 32 & 7 & 4 & 2 \\
-49 & -83 & -3 & -6 & -3 & -2 & -9 & -25 & -56 & -38 & -18 & -7 \\
95 & 61 & 7 & 4 & 3 & 2 & 8 & 21 & 45 & 63 & 28 & 10 \\
-63 & -17 & -5 & -1 & -1 & -1 & -3 & -7 & -14 & -39 & -17 & -6
\end{array}\right)$}\]

\subsubsection*{Case 10} $\alpha= \left(0,0,0,\frac{1}{4},\frac{3}{4}\right)$, $\beta=\left(\frac{1}{2},\frac{1}{2},\frac{1}{2},\frac{1}{3},\frac{2}{3}\right)$, and $B$ has infinite order.
\[\adjustbox{max width=\textwidth}{%
$M = \left(\begin{array}{rrrrrrrrrrrrrrrr}
-1 & -1 & -1 & -5 & -1 & -40 & -8944 & -8944 & -40 & -1 & -23 & -54 & -97 & -97 & -54 & -23 \\
-3 & -3 & 2 & -17 & 19 & 119 & 20877 & 14600 & -60 & 36 & 68 & 139 & 237 & -257 & -119 & -38 \\
-4 & -9 & -2 & -24 & -22 & -141 & -56171 & -68806 & -418 & -27 & -56 & -148 & -249 & -919 & -538 & -226 \\
-3 & -5 & 2 & -19 & 20 & 138 & 6198 & -6437 & -139 & 15 & 65 & 160 & 240 & -430 & -230 & -105 \\
-1 & -6 & -1 & -7 & -40 & -100 & -50376 & -56653 & -279 & -23 & -54 & -97 & -131 & -625 & -355 & -160
\end{array}\right)$}\]

\subsubsection*{Case 11} $\alpha= \left(0,0,0,\frac{1}{6},\frac{5}{6}\right)$, $\beta=\left(\frac{1}{2},\frac{1}{2},\frac{1}{2},\frac{1}{3},\frac{2}{3}\right)$, and $B$ has infinite order.
\[\adjustbox{max width=\textwidth}{%
$M =\left(\begin{array}{rrrrrrrrrrrrrrr}
-1 & 0 & -3 & -61 & -223 & -465 & -715 & -715 & -465 & -223 & -61 & -3 & -35 & -1 & -1 \\
-3 & 1 & 11 & 241 & 831 & 1637 & 2395 & -1935 & -1145 & -427 & -21 & 49 & -121 & -3 & 3 \\
-4 & -2 & 0 & -416 & -1320 & -2424 & -3368 & -7586 & -5190 & -2706 & -854 & -42 & -172 & -10 & -4 \\
-3 & 2 & 21 & 427 & 1145 & 1935 & 2581 & -1637 & -831 & -241 & -11 & -21 & -137 & -3 & 3 \\
-1 & -1 & -61 & -223 & -465 & -715 & -925 & -5255 & -3497 & -1723 & -485 & -23 & -51 & -7 & -1
\end{array}\right)$}\]

\subsubsection*{Case 12} $\alpha= \left(0,0,0,\frac{1}{6},\frac{5}{6}\right)$, $\beta=\left(\frac{1}{2},\frac{1}{5},\frac{2}{5},\frac{3}{5},\frac{4}{5}\right)$, and $B$ has finite order.
\[\adjustbox{max width=\textwidth}{%
$M =\left(\begin{array}{rrrrrrrrrrrrrrrrrrrr}
-1 & -1 & -1 & -27 & -64 & -99 & -121 & -132 & -132 & -121 & -99 & -64 & -27 & -18 & -18 & -1 \\
-1 & 3 & 16 & 107 & 229 & 332 & 385 & 407 & -119 & -110 & -77 & -29 & 10 & 6 & 71 & 25 \\
-5 & -4 & -32 & -173 & -341 & -464 & -515 & -539 & -673 & -625 & -550 & -403 & -211 & -116 & -101 & -46 \\
2 & 3 & 44 & 157 & 275 & 352 & 383 & 409 & 275 & 273 & 266 & 213 & 119 & 65 & 80 & 30 \\
-5 & -1 & -27 & -64 & -99 & -121 & -132 & -145 & -671 & -627 & -530 & -357 & -161 & -107 & -42 & -18
\end{array}\right)$}\]

\subsection{Cases of type $O(4,1)$}
\subsubsection*{Case 1} $\alpha= \left(0,0,0,\frac{1}{3},\frac{2}{3} \right)$, $\beta=\left(\frac{1}{2},\frac{1}{2},\frac{1}{2},\frac{1}{4},\frac{3}{4} \right)$, and $B$ has infinite order. 
\[\adjustbox{max width=\textwidth}{%
$M =\left(\begin{array}{rrrrrrrrrrr}
-1 & -1 & -1 & 0 & -1 & -1 & -1 & -1 & -3 & -3 & -1 \\
-2 & -2 & 1 & 1 & 2 & -3 & 1 & -1 & 5 & 1 & -2 \\
-2 & -5 & 0 & -1 & 0 & -3 & 1 & -5 & -5 & -9 & -4 \\
-2 & -4 & 1 & 1 & 0 & -3 & 1 & -5 & -3 & -7 & -4 \\
-1 & -4 & -1 & -1 & -1 & -2 & -2 & -4 & -10 & -14 & -5
\end{array}\right) $}\]

\subsubsection*{Case 2} $\alpha= \left(0,0,0,\frac{1}{3},\frac{2}{3} \right)$, $\beta=\left(\frac{1}{2},\frac{1}{2},\frac{1}{2},\frac{1}{6},\frac{5}{6} \right)$, and $B$ has infinite order.
\[\adjustbox{max width=\textwidth}{%
$M = \left(\begin{array}{rrrrrrrrrrr}
-1 & -1 & -1 & 0 & -1 & -2 & -1 & -1 & -3 & -3 & -1 \\
-1 & -1 & 1 & 1 & 2 & -3 & 1 & 0 & 4 & -1 & -1 \\
0 & -2 & 0 & -1 & 0 & 0 & 1 & -2 & -6 & -5 & -1 \\
-1 & -1 & 1 & 1 & 0 & -2 & 1 & -2 & 2 & 3 & -1 \\
-1 & -3 & -1 & -1 & -1 & -3 & -2 & -3 & -5 & -10 & -4
\end{array}\right)$}\]
\subsubsection*{Case 3} $\alpha= \left(0,0,0,\frac{1}{3},\frac{2}{3} \right)$, $\beta=\left(\frac{1}{2}, \frac{1}{5},\frac{2}{5},\frac{3}{5},\frac{4}{5} \right)$, and $B$ has finite order.
\[\adjustbox{max width=\textwidth}{%
$M =\left(\begin{array}{rrrrrrrrrrrrr}
-1 & -1 & 0 & -1 & -1 & -3 & -3 & -1 & -2 & -2 & -2 & -2 & -1 \\
-1 & 1 & 1 & 2 & 1 & 0 & 3 & 0 & 3 & 2 & 2 & 0 & -1 \\
-3 & 0 & -1 & 0 & 1 & -8 & -6 & -3 & -3 & -2 & -5 & -3 & -2 \\
-2 & 1 & 1 & 0 & 1 & 0 & 2 & -3 & -2 & -1 & -1 & 1 & -2 \\
-3 & -1 & -1 & -1 & -2 & -9 & -6 & -3 & -6 & -7 & -4 & -6 & -4
\end{array}\right) $}\]
\subsubsection*{Case 4} $\alpha= \left(0,0,0,\frac{1}{4},\frac{3}{4}\right)$, $\beta=\left(\frac{1}{2},\frac{1}{3},\frac{2}{3},\frac{1}{6},\frac{5}{6} \right)$, and $B$ has finite order.
\[\adjustbox{max width=\textwidth}{%
$M =\left(\begin{array}{rrrrrrrrrrrrr}
-1 & -1 & 0 & -1 & -1 & -1 & -1 & -2 & -2 & -2 & -2 & -1 & -1 \\
0 & 2 & 1 & 3 & 2 & 2 & 1 & 5 & 4 & 4 & 3 & 0 & 0 \\
-3 & -2 & -2 & -3 & -1 & -2 & -4 & -8 & -7 & -9 & -8 & -2 & -3 \\
1 & 2 & 2 & 2 & 1 & 3 & 1 & 5 & 6 & 6 & 7 & 0 & 2 \\
-3 & -1 & -1 & -1 & -1 & -2 & -3 & -6 & -7 & -5 & -6 & -3 & -4
\end{array}\right) $}\]
\subsubsection*{Case 5} $\alpha= \left(0,0,0,\frac{1}{4},\frac{3}{4}\right)$, $\beta=\left(\frac{1}{2}, \frac{1}{5},\frac{2}{5},\frac{3}{5},\frac{4}{5} \right)$, and $B$ has finite order.
\[\adjustbox{max width=\textwidth}{%
$M = \left(\begin{array}{rrrrrrrrrrr}
-1 & -1 & 0 & -1 & -1 & -1 & -1 & -3 & -3 & -1 & -1 \\
-1 & 2 & 1 & 3 & 2 & 2 & 0 & 8 & 4 & -1 & -1 \\
-4 & -2 & -2 & -3 & -1 & -2 & -5 & -13 & -14 & -3 & -4 \\
0 & 2 & 2 & 2 & 1 & 3 & 0 & 8 & 7 & -1 & 1 \\
-4 & -1 & -1 & -1 & -1 & -2 & -4 & -10 & -14 & -4 & -5
\end{array}\right)$}\]
\subsubsection*{Case 6} $\alpha= \left(0,0,0,\frac{1}{4},\frac{3}{4}\right)$, $\beta=\left(\frac{1}{2},\frac{1}{8},\frac{3}{8},\frac{5}{8}, \frac{7}{8} \right)$, and $B$ has finite order.
\[\adjustbox{max width=\textwidth}{%
$M =\left(\begin{array}{rrrrrrrrrrrrrrr}
-1 & -1 & 0 & -6 & -5 & -4 & -4 & -6 & -6 & -8 & -15 & -15 & -8 & -5 & -6 \\
0 & 2 & 6 & 18 & 9 & 7 & 6 & 14 & 10 & 18 & 37 & -1 & 7 & -1 & -1 \\
-2 & -2 & -13 & -18 & -2 & -7 & -14 & -23 & -15 & -34 & -42 & -26 & -31 & -2 & -11 \\
2 & 2 & 13 & 11 & 2 & 14 & 7 & 15 & 23 & 31 & 26 & 42 & 34 & 2 & 18 \\
-3 & -1 & -6 & -5 & -4 & -10 & -11 & -16 & -20 & -15 & -14 & -52 & -26 & -14 & -24
\end{array}\right) $}\]
\subsubsection*{Case 7} $\alpha= \left(0,0,0,\frac{1}{4},\frac{3}{4}\right)$, $\beta=\left(\frac{1}{2},\frac{1}{12},\frac{5}{12},\frac{7}{12}, \frac{11}{12}\right)$, and $B$ has finite order.
\[\adjustbox{max width=\textwidth}{%
$M =\left(\begin{array}{rrrrrrrrrrrrrrrrrrrrrrrrr}
-1 & -1 & 0 & -40 & -31 & -22 & -22 & -26 & -26 & -52 & -99 & -72 & -32 & -99 & -188 & -188 & -99 & -32 & -72 & -99 & -52 & -31 & -102 & -102 & -40 \\
0 & 2 & 40 & 120 & 53 & 35 & 40 & 56 & 26 & 130 & 245 & 117 & 24 & 89 & -34 & 465 & 241 & 42 & -40 & -27 & 47 & -9 & 275 & 53 & -9 \\
-1 & -2 & -89 & -120 & -4 & -35 & -62 & -95 & -47 & -212 & -266 & -43 & -11 & -311 & -144 & -511 & -420 & -53 & 50 & -85 & -173 & 27 & -399 & -217 & -31 \\
3 & 2 & 89 & 71 & 4 & 84 & 57 & 73 & 121 & 225 & 184 & 22 & 85 & 519 & 699 & 332 & 410 & 43 & 115 & 365 & 264 & 35 & 319 & 501 & 160 \\
-3 & -1 & -40 & -31 & -22 & -62 & -57 & -52 & -82 & -99 & -72 & -32 & -74 & -340 & -653 & -154 & -188 & -56 & -189 & -344 & -182 & -84 & -155 & -377 & -160
\end{array}\right) $}\]
\subsubsection*{Case 8} $\alpha= \left(0,0,0,\frac{1}{6},\frac{5}{6}\right)$, $\beta=\left(\frac{1}{2},\frac{1}{8},\frac{3}{8},\frac{5}{8}, \frac{7}{8} \right)$, and $B$ has finite order.
\[\adjustbox{max width=\textwidth}{%
$M =\left(\begin{array}{rrrrrrrrrrrrrrrrr}
-1 & -1 & 0 & -3 & -4 & -3 & -2 & -3 & -3 & -7 & -7 & -2 & -7 & -7 & -3 & -4 & -3 \\
0 & 3 & 3 & 12 & 13 & 8 & 5 & 10 & 3 & 15 & 26 & 1 & 26 & 13 & -1 & -1 & 1 \\
-3 & -4 & -8 & -18 & -16 & -8 & -6 & -16 & -15 & -43 & -48 & -6 & -48 & -41 & -5 & -10 & -13 \\
4 & 3 & 8 & 13 & 10 & 5 & 6 & 15 & 16 & 48 & 43 & 6 & 41 & 48 & 8 & 16 & 18 \\
-4 & -1 & -3 & -4 & -3 & -2 & -3 & -6 & -13 & -33 & -22 & -7 & -20 & -33 & -11 & -17 & -15
\end{array}\right)$}\]
\subsubsection*{Case 9} $\alpha= \left(0,0,0,\frac{1}{6},\frac{5}{6}\right)$, $\beta=\left(\frac{1}{2},\frac{1}{12},\frac{5}{12},\frac{7}{12}, \frac{11}{12}\right)$, and $B$ has finite order.
\[\adjustbox{max width=\textwidth}{%
$M = \left(\begin{array}{rrrrrrrrrrrrrrrrrrr}
-1 & -1 & 0 & -3 & -4 & -3 & -2 & -3 & -3 & -15 & -15 & -2 & -7 & -7 & -3 & -13 & -13 & -4 & -3 \\
0 & 3 & 3 & 12 & 13 & 8 & 5 & 10 & 3 & 34 & 53 & 1 & 26 & 11 & -1 & 49 & 20 & -1 & 1 \\
-2 & -4 & -8 & -18 & -16 & -8 & -6 & -16 & -12 & -96 & -110 & -4 & -48 & -36 & -2 & -90 & -68 & -6 & -10 \\
5 & 3 & 8 & 13 & 10 & 5 & 6 & 15 & 19 & 125 & 111 & 8 & 43 & 55 & 11 & 81 & 103 & 20 & 21 \\
-4 & -1 & -3 & -4 & -3 & -2 & -3 & -6 & -13 & -68 & -49 & -7 & -18 & -33 & -11 & -33 & -62 & -17 & -15
\end{array}\right)$}\]

\section*{Acknowledgements}
JB is supported through a fellowship from MPIM, Bonn. MN is supported by DFG grant 281869850 (RTG 2229, ``Asymptotic Invariants and Limits of Groups and Spaces'').

\nocite{}
\bibliographystyle{abbrv}
\bibliography{BDN}
\end{document}